\documentclass[a4paper,11pt]{amsart}

\usepackage[utf8]{inputenc}
\usepackage[T1]{fontenc}
\usepackage{latexsym}
\usepackage{enumitem}
\usepackage{amsxtra}
\usepackage{amsmath}
\usepackage{amsfonts}
\usepackage{amssymb}
\usepackage{blkarray}
\usepackage{multirow}
\usepackage{mathrsfs}
\usepackage{mathdots}
\usepackage{pdfpages}
\usepackage{cite}
\usepackage{comment}
\usepackage{hyperref}
\usepackage[capitalise]{cleveref}

\newtheorem{theorem}{Theorem}[section]
\newtheorem{lemma}[theorem]{Lemma}
\newtheorem{corollary}[theorem]{Corollary}

\newtheorem{sublemma}{}[theorem]
\newtheorem{conjecture}[theorem]{Conjecture}

\newenvironment{subproof}[1][\proofname]{%
  \begin{proof}[#1]%
}{%
  \end{proof}%
}
\numberwithin{equation}{section}

\theoremstyle{definition}
\newtheorem{definition}[theorem]{Definition}

\theoremstyle{remark}

\newcommand{\seq}[1]{[#1]}
\newcommand{\ba}{\backslash}

\DeclareMathOperator{\cl}{cl}

\setenumerate{label=\rm(\roman*),midpenalty=2}

\begin{document}
\title{On a generalisation of spikes}

\author[N.~Brettell]{Nick Brettell}
\address{School of Mathematics and Statistics\\
  Victoria University of Wellington\\
New Zealand}
\email{nbrettell@gmail.com}

\author[R.~Campbell]{Rutger Campbell}
\address{Department of Combinatorics and Optimization\\
  University of Waterloo\\
Canada}
\email{rtrjvdc@gmail.com}

\author[D.~Chun]{Deborah Chun}
\address{Mathematics Department\\West Virginia University Institute of Technology\\410 Neville Street\\Beckley, West Virginia 25801\\USA}
\email{deborah.chun@mail.wvu.edu}

\author[K.~Grace]{Kevin Grace}
\address{Department of Mathematics\\
  Louisiana State University\\
Baton Rouge, Louisiana\\USA}
\email{kgrace3@lsu.edu}

\author[G.~Whittle]{Geoff Whittle}
\address{School of Mathematics and Statistics\\
  Victoria University of Wellington\\
New Zealand}
\email{geoff.whittle@vuw.ac.nz}

\thanks{The authors would like to thank  the MAThematical Research Institute (MATRIx), Creswick, Victoria, Australia, for support and hospitality during the \href{https://www.matrix-inst.org.au/events/tutte-centenary-retreat/}{Tutte Centenary Retreat}, 26 Nov -- 2 Dec 2017, where work on this paper was initiated. The first and fifth authors were supported by the New Zealand Marsden Fund. The second author was supported by NSERC Scholarship PGSD3-489418-2016. The fourth author was supported by National Science Foundation grant 1500343.}

\subjclass{05B35}
\date{\today}

\begin{abstract}
  We consider matroids with the property that every subset of the ground set of size~$t$ is contained in both an $\ell$-element circuit and an $\ell$-element cocircuit; we say that such a matroid has the \emph{$(t,\ell)$-property}.
  We show that for any positive integer $t$, there is a finite number of matroids with the $(t,\ell)$-property for $\ell<2t$; however, matroids with the $(t,2t)$-property form an infinite family.
  We say a matroid is a \emph{$t$-spike} if there is a partition of the ground set into pairs such that the union of any $t$ pairs is a circuit and a cocircuit.
  Our main result is that if a sufficiently large matroid has the $(t,2t)$-property, then it is a $t$-spike.
  Finally, we present some properties of $t$-spikes.
\end{abstract}

\maketitle

\section{Introduction}

For all $r \ge 3$, a rank-$r$ \emph{spike} is a matroid on $2r$ elements with a partition $(X_1,X_2,\dotsc,X_r)$ into pairs such that $X_i \cup X_j$ is a circuit and a cocircuit for all distinct $i,j\in \{1,2,\dotsc,r\}$.
Spikes frequently arise in the matroid theory literature (see, for example, \cite{ovw1996,seymour1981,ggw2002,geelen2008}) as a seemingly benign, yet wild, class of matroids.
Miller~\cite{miller2014} proved that if $M$ is a sufficiently large matroid 
having the property that every two elements share both a $4$-element circuit and a $4$-element cocircuit, then $M$ is a spike.

We consider generalisations of this result. 
We say that a matroid~$M$ has the \emph{$(t,\ell)$-property} if every $t$-element subset of $E(M)$ is contained in both an $\ell$-element circuit and an $\ell$-element cocircuit.
It is well known that
the only matroids with the $(1,3)$-property are wheels and whirls, and
Miller's result shows that if $M$ is a sufficiently large matroid with the $(2,4)$-property, then $M$ is a spike.

We first show that when $\ell < 2t$, there are only finitely many matroids with the $(t,\ell)$-property.
However, for any positive integer $t$, the matroids with the $(t,2t)$-property form an infinite class:
when $t=1$, this is the class of matroids obtained by taking direct sums of copies of $U_{1,2}$;
when $t = 2$, the class contains the infinite family of spikes.
Our main result is the following:
\begin{theorem}
  \label{mainthm}
  There exists a function $f$ such that if $M$ is a matroid with the $(t,2t)$-property, and $|E(M)| \ge f(t)$, then $E(M)$ has a partition into pairs such that the union of any $t$ pairs is both a circuit and a cocircuit.
\end{theorem}

\noindent
We call a matroid with such a partition a \emph{$t$-spike}.
(A traditional spike is a $2$-spike.  Note also that what we call a spike is sometimes referred to as a \emph{tipless spike}.)

We also prove some properties of $t$-spikes, which demonstrate that $t$-spikes are highly structured matroids.
In particular, a $t$-spike has $2r$ elements for some positive integer~$r$, it has rank~$r$ (and corank~$r$), any circuit that is not a union of $t$ pairs avoids at most $t-2$ of the pairs, and any sufficiently large $t$-spike is $(2t-1)$-connected.
We show that a $t$-spike's partition into pairs describes crossing $(2t-1)$-separations in the matroid; that is, an appropriate concatenation of this partition is a $(2t-1)$-flower (more specifically, a $(2t-1)$-anemone), following the terminology of~\cite{ao2008}.
We also describe a construction of a $(t+1)$-spike from a $t$-spike, 
and show that every $(t+1)$-spike can be obtained from some $t$-spike in this way.

Our methods in this paper are extremal, so the lower bounds on $|E(M)|$ that we obtain, given by the function $f$, are extremely large, and we make no attempts to optimise these.  For $t=2$, Miller~\cite{miller2014} showed that $f(2)=13$ is best possible, and he described the other matroids with the $(2,4)$-property when $|E(M)| \le 12$.  We see no reason why a similar analysis could not be undertaken for, say, $t=3$.

There are a number of interesting variants of the $(t,\ell)$-property.
In particular, 
we say that a matroid has the \emph{$(t_1,\ell_1,t_2,\ell_2)$-property} if
every $t_1$-element set is contained in an $\ell_1$-element circuit, and every $t_2$-element set is contained in an $\ell_2$-element cocircuit.
Although we focus here on the case where $t_1=t_2$ and $\ell_1=\ell_2$, we show, in \cref{sec:smallL}, that there are only finitely many matroids with the $(t_1,\ell_1,t_2,\ell_2)$-property when 
$\ell_1 < 2t_1$ or $\ell_2 < 2t_2$.
Oxley et al.~\cite{pfeil} recently considered the case where $(t_1,\ell_1,t_2,\ell_2)=(2,4,1,k)$ and $k\in\{3,4\}$.
In particular, they proved, for $k \in \{3,4\}$, that a $k$-connected matroid~$M$ with $|E(M)|\ge k^2$ has the $(2,4,1,k)$-property if and only if $M \cong M(K_{k,n})$ for some $n \ge k$.
This gives credence to the idea that sufficiently large matroids with the $(t_1,\ell_1,t_2,\ell_2)$-property, for appropriate values of $t_1,\ell_1,t_2,\ell_2$, may form structured classes.  In particular, we conjecture the following generalisation of \cref{mainthm}:

\begin{conjecture}
  There exists a function $f(t_1,t_2)$ such that
  if $M$ is a matroid with the $(t_1,2t_1,t_2,2t_2)$-property, for positive integers $t_1$ and $t_2$, and $|E(M)| \ge f(t_1,t_2)$, 
  then $E(M)$ has a partition into pairs such that the union of any $t_1$ pairs is a circuit, and the union of any $t_2$ pairs is a cocircuit.
\end{conjecture}




The study of matroids with the $(t,2t)$-property was motivated by problems in matroid connectivity.
Tutte proved that wheels and whirls (that is, matroids with the $(1,3)$-property) are the only $3$-connected matroids with no element whose deletion or contraction preserves $3$-connectivity~\cite{tutte1966}.
Moreover, spikes (matroids with the $(2,4)$-property) are the only $3$-connected matroids with $|E(M)| \ge 13$ having no triangles or triads, and no pair of elements whose deletion or contraction preserves $3$-connectivity~\cite{williams2015}.
We envision that $t$-spikes could also play a role in a connectivity ``chain theorem'': they are $(2t-1)$-connected matroids, having no circuits or cocircuits of size $(2t-1)$, with the property that for every $t$-element subset $X \subseteq E(M)$, neither $M/X$ nor $M \ba X$ is $(t+1)$-connected.
We conjecture the following:
\begin{conjecture}
  There exists a function $f(t)$ such that if $M$ is a $(2t-1)$-connected matroid with no circuits or cocircuits of size $2t-1$, and $|E(M)| \ge f(t)$, then either
  \begin{enumerate}
    \item there exists a $t$-element set $X \subseteq E(M)$ such that either $M/X$ or $M \ba X$ is $(t+1)$-connected, or
    \item $M$ is a $t$-spike.
  \end{enumerate}
\end{conjecture}


This paper is structured as follows.
In \cref{sec:smallL}, we prove that there are only finitely many matroids with the $(t,\ell)$-property, for $\ell < 2t$.
In \cref{sec:echidnas}, we define $t$-echidnas and $t$-spikes, and show that a matroid with the $(t,2t)$-property and having a sufficiently large $t$-echidna is a $t$-spike.
We prove \cref{mainthm} in \cref{sec:t2t}.
Finally, we present some properties of $t$-spikes in \cref{sec:tspikeprops}.

\section{Preliminaries}

Our notation and terminology follow Oxley~\cite{oxbook}.
We refer to the fact that a circuit and a cocircuit cannot intersect in exactly one element as ``orthogonality''.
We say that a $k$-element set is a \emph{$k$-set}.  A set~$S_1$ \emph{meets} a set~$S_2$ if $S_1 \cap S_2 \neq \emptyset$.
We denote $\{1,2,\dotsc,n\}$ by $\seq{n}$, and, for positive integers $i < j$, we denote $\{i,i+1,\dotsc,j\}$ by $[i,j]$.
We denote the set of positive integers by $\mathbb{N}$.

\begin{lemma}
  \label{lem:Ramsey}
  There exists a function $f : \mathbb{N} \times \mathbb{N} \rightarrow \mathbb{N}$ such that, if $\mathcal{S}$ is a collection of distinct $s$-sets and $|\mathcal{S}|\geq f(s,n)$, then there is some $\mathcal{S}'\subseteq\mathcal{S}$ with $|\mathcal{S}'|=n$, and a set $J$ with $0\le |J|<s$, such that $S_1\cap S_2=J$ for all distinct $S_1,S_2\in\mathcal{S}'$.
\end{lemma}

\begin{proof}
  We define $f(1,n)=n$, and $f(s,n)=s(n-1)f(s-1,n)$ for $s > 1$.
  Note that $f$ is increasing.
  We claim that this function satisfies the lemma.
  We proceed by induction on $s$.
  If $s=1$, then the claim holds with $J=\emptyset$.

  Let $\mathcal{S}$ be a collection of $s$-sets with $|\mathcal{S}|\geq f(s,n)$.
  Suppose there are $n$ pairwise disjoint sets in $\mathcal{S}$.
  Then the desired conditions are satisfied if we take $J=\emptyset$.
  Thus, we may assume that there is some maximal $\mathcal{D} \subseteq \mathcal{S}$ consisting of pairwise disjoint sets, with $|\mathcal{D}|\leq n-1$.
  Each $S \in \mathcal{S}-\mathcal{D}$ meets some $D \in \mathcal{D}$.
  Each such $D$ has $s$ elements.
  Therefore, each $S \in \mathcal{S}$ contains at least one of $(n-1)s$ elements $e\in\cup \mathcal{D}$.
  By the pigeonhole principle, there is some $e\in\cup\mathcal{D}$ such that
  $$|\{S\in\mathcal{S}:e\in S\}|\geq\frac{f(s,n)}{(n-1)s}= f(s-1,n).$$

  Let $\mathcal{T}=\{S-\{e\}:e\in S\in\mathcal{S}\}$. Then, for every $T\in\mathcal{T}$, we have $|T|=s-1$.
  Moreover, $|\mathcal{T}|=|\{S\in\mathcal{S}:e\in S\}|\geq f(s-1,n)$.
  By the induction assumption, there is a subset $\mathcal{T}'\subseteq\mathcal{T}$ with $|\mathcal{T}'|=n$ and a set $J'$, with $|J'|<s-1$ such that $T_1\cap T_2=J'$ for all distinct $T_1,T_2\in\mathcal{T}'$. Let $\mathcal{S}'=\{T\cup \{e\}: T\in\mathcal{T}'\}$.
  Then, $\mathcal{S}'\subseteq\mathcal{S}$ with $|\mathcal{S}'|=n$ such that $S_1\cap S_2=J'\cup\{e\}$ for all distinct $S_1,S_2\in\mathcal{S}'$, and $|J\cup\{e\}| < s$.
\end{proof}

\section{Matroids with the $(t,\ell)$-property for $\ell < 2t$}
\label{sec:smallL}

Recall that a matroid has the $(t_1,\ell_1,t_2,\ell_2)$-property if every $t_1$-element set is contained in an $\ell_1$-element circuit, and every $t_2$-element set is contained in an $\ell_2$-element cocircuit.
In this section, we prove that there are only finitely many matroids with the $(t_1,\ell_1,t_2,\ell_2)$-property if $\ell_2 < 2t_2$.
By duality, the same is true if $\ell_1 < 2t_1$.
As a special case, we have that there are only finitely many matroids with the $(t,\ell)$-property for $\ell < 2t$.

\begin{lemma}
  \label{lem:trapping}
  Let $\mathcal{C}$ be a collection of circuits of a matroid~$M$ such that, for some $J \subseteq E(M)$ with $|J|\leq k$, we have $C\cap C'=J$, for all distinct $C,C'\in\mathcal{C}$.
  Then, for every subcollection $\{C_1,\dotsc,C_{2^k}\}\subseteq\mathcal{C}$ of size $2^k$, there is a circuit contained in $\bigcup_{i=1}^{2^k}C_i-J$.
\end{lemma}

\begin{proof}
  We may assume $|\mathcal{C}|\geq2^k$; otherwise, the result holds vacuously. Also, we may assume $k>0$ as the result holds for any singleton subcollection of $\mathcal{C}$ with $J=\emptyset$. Therefore, $\mathcal{C}$ has at least one subcollection $\mathcal{C}' = \{C_1,\ldots C_{2^k}\}$, with $|\mathcal{C}'|= 2^k \geq2$.

  Let $x_1,x_2,\ldots, x_{|J|}$ be the elements of $J$.
  Define $Z_{i,0}=C_i$, for $i \in [2^k]$, and recursively define $Z_{i,j}=Z_{2i-1,j-1}\cup Z_{2i,j-1}$ for 
  $j \in [k]$ and $i \in [2^{k-j}]$.
  Note that each $Z_{i,j}$ is the union of $2^j$ members of $\mathcal{C}$.
  We will show, by induction on $j$, that $Z_{i,j}-\{x_1,x_2,\ldots,x_j\}$ contains a circuit.
  This is clear when $j=0$.
  Now let $j \ge 1$.
  By the induction hypothesis, $Z_{2i-1,j-1}$ and $Z_{2i,j-1}$ each contain a circuit, $C'_1$ and $C'_2$ respectively, disjoint from $\{x_1,x_2,\ldots,x_{j-1}\}$, for each $i \in [2^{k-j}]$.
  (Moreover, $C'_1\neq C'_2$ since $C'_1\cap C'_2\subseteq Z_{2i-1,j-1}\cap Z_{2i,j-1}\subseteq J$, which is independent since $J$ is the intersection of at least two 
  circuits.) We may assume that neither $Z_{2i-1,j-1}$ nor $Z_{2i,j-1}$ contains a circuit disjoint from $\{x_1,x_2,\ldots,x_j\}$; otherwise, so does $Z_{i,j}$.
  Thus, $C'_1$ and $C'_2$ both contain $x_j$.
  By circuit elimination, there is a circuit~$C'_3$ contained in $(C'_1\cup C'_2)-\{x_j\}\subseteq Z_{i,j}-\{x_1,x_2,\ldots,x_j\}$.
  This completes the induction argument.
  In particular, there is a circuit contained in $Z_{1,k}-\{x_1,x_2,\ldots, x_{|J|}\}=\bigcup_{i=1}^{2^k}C_i-J$, as required.
\end{proof}

\begin{lemma}
  \label{lem:disjoint2}
  There exists a function $g: \mathbb{N} \times \mathbb{N} \rightarrow \mathbb{N}$ such that if $M$ has at least $g(\ell,d)$-many $\ell$-element circuits, then $M$ has a collection of $d$ pairwise disjoint circuits.
\end{lemma}
\begin{proof}
  Let $\mathcal{C}$ be the collection of $\ell$-element circuits of $M$, let $f$ be the function of \cref{lem:Ramsey}, and let $g(\ell,d) = f(\ell,2^{\ell-1}d)$.
  Then, by \cref{lem:Ramsey}, there is a subset $\mathcal{C}'\subseteq\mathcal{C}$ with $|\mathcal{C}'|=2^{\ell-1}d$, and a set $J$, with $0\leq|J|\le \ell-1$, such that $C\cap C'=J$ for every pair $C,C'\in\mathcal{C}'$.
  Say $\mathcal{C}'=\{C_1,C_2,\ldots,C_{2^{\ell-1}d}\}$.

  If $J=\emptyset$, then $M$ has $2^{\ell-1}d \ge d$ pairwise disjoint circuits, as required.
  Thus, we may assume that $J\neq\emptyset$.
  For each $C_i\in\mathcal{C}'$, let $D_i=C_i-J$, and observe that the $D_i$'s are pairwise disjoint.
  For $j \in [d]$, let \[D_j'=\bigcup \limits_{i=1}^{2^{\ell-1}} D_{(j-1)(2^{\ell-1})+i}.\]
  By \cref{lem:trapping}, each $D_j'$ contains a circuit~$C_j'$, and the $C_j'$'s are pairwise disjoint.
\end{proof}

\begin{theorem}
\label{thm:2tstrong}
  Let $t_1$, $\ell_1$, $t_2$, and $\ell_2$ be positive integers.
  If $\ell_1 < 2t_1$, then there is a finite number of matroids with the $(t_1,\ell_1,t_2,\ell_2)$-property.
  By duality, the same is true if $\ell_2 < 2t_2$.
\end{theorem}

\sloppy
\begin{proof}
  It suffices to prove the result when $\ell_2 < 2t_2$.
  So let $\ell_2 < 2t_2$, and let $g$ be the function given in \cref{lem:disjoint2}.

  Suppose $M$ has at least $g(\ell_1,t_2)$-many $\ell_1$-element circuits. By \cref{lem:disjoint2}, $M$ has a collection of $t_2$ pairwise disjoint circuits. Call this collection $\mathcal{C} = \{C_1,\ldots,C_{t_2}\}$. 
  Let $b_i$ be an element of $C_i$, for each $i \in [t_2]$.
  By the $(t_1,\ell_1,t_2,\ell_2)$-property, there is an $\ell_2$-element cocircuit~$C^*$ containing $\{b_1,\ldots,b_{t_2}\}$.
  By orthogonality, for each $i \in [t_2]$ there is an element $b'_i\neq b_i$ such that $b'_i\in C_i\cap C^*$.
  This implies that $\ell_2=|C^*|\geq 2t_2$; a contradiction. Thus, 
  $M$ has fewer than $g(\ell_1,t_2)$-many $\ell_1$-element circuits.

  Suppose $|E(M)|\geq \ell_1\cdot g(\ell_1,t_2)$.
  Partition a subset of $E(M)$ into $\lfloor \ell_1/t_1\rfloor \cdot g(\ell_1,t_2)$ pairwise disjoint $t_1$-sets.
  By the $(t_1,\ell_1,t_2,\ell_2)$-property, each of these $t_1$-sets is contained in an $\ell_1$-element circuit.
  The collection consisting of these $\ell_1$-element circuits contains at least $g(\ell_1,t_2)$ distinct circuits.
  This contradicts the fact that $M$ has fewer than $g(\ell_1,t_2)$-many $\ell_1$-element circuits.
  Therefore, $|E(M)| < \ell_1\cdot g(\ell_1,t_2)$.
  The result follows.
\end{proof}
\fussy

Note that there may still be infinitely many matroids where every $t_1$-element set is in an $\ell_1$-element circuit for fixed $\ell_1 < 2t_1$; 
it is necessary that the matroids in \cref{thm:2tstrong} have the property that every $t_2$-element set is in an $\ell_2$-element cocircuit, for fixed $t_2$ and $\ell_2$.
To see this, observe that projective geometries on at least three elements form an infinite family of matroids with the property that every pair of elements is in a $3$-element circuit.

\begin{corollary}
  \label{thm:2t}
  Let $t$ and $\ell$ be positive integers.
  When $\ell<2t$, there is a finite number of matroids with the $(t,\ell)$-property.
\end{corollary}

\section{Echidnas and $t$-spikes}
\label{sec:echidnas}

We now focus on matroids with the $(t,2t)$-property.
In \cref{sec:t2t}, we will show that every sufficiently large matroid with the $(t,2t)$-property has a partition into pairs such that the union of any $t$ of these pairs is both a circuit and a cocircuit.
We call such a matroid a $t$-spike.
We first define a related structure: a $t$-echidna.

\begin{definition}
  \label{def:echidna}
  Let $M$ be a matroid.
  A $t$-\emph{echidna} of order $n$ is a partition $(S_1,\ldots, S_n)$ of a subset of $E(M)$ such that 
  \begin{enumerate}
    \item $|S_i|=2$ for all $i \in \seq{n}$, and 
    \item $\bigcup_{i \in I}S_i$ is a circuit for all $I \subseteq \seq{n}$ with $|I|=t$.
  \end{enumerate}
  For $i \in \seq{n}$, we say $S_i$ is a \emph{spine}.
  We say $(S_1,\ldots,S_n)$ is a \emph{$t$-coechidna} of $M$ if $(S_1,\ldots,S_n)$ is a $t$-echidna of $M^*$.
\end{definition}

\begin{definition}
  \label{def:tspike}
  A matroid~$M$ is a \emph{$t$-spike} of order~$r$ if there exists a partition $\pi=(A_1,\ldots,A_r)$ of $E(M)$ such that $\pi$ is a $t$-echidna and a $t$-coechidna, for some $r \ge t$.
  We say $\pi$ is the \emph{associated partition} of the $t$-spike~$M$, and $A_i$ is an \emph{arm} of the $t$-spike for each $i \in \seq{r}$.
\end{definition}

\noindent
Note that if $M$ is a $t$-spike, then $M^*$ is a $t$-spike.

\medskip

In this section, we prove, as \cref{lem:swamping}, that if $M$ is a matroid with the $(t,2t)$-property, and $M$ has a $t$-echidna of order $4t-3$, then $M$ is a $t$-spike.

\begin{lemma}
  \label{lem:coechidna}
  Let $M$ be a matroid with the $(t,2t)$-property.
  If $M$ has a $t$-echidna $(S_1,\ldots, S_n)$, where $n\geq3t-1$, then $(S_1,\ldots, S_n)$ is also a $t$-coechidna of $M$.
\end{lemma}

\begin{proof}
  Let $S_i=\{x_i,y_i\}$ for each $i \in [n]$.
  By definition, if $J$ is a $t$-element subset of $[n]$, then $\bigcup_{j \in J} S_j$ is a circuit.
  Consider such a circuit~$C$; without loss of generality, we let $C=\{x_1,y_1,\ldots,x_t,y_t\}$.
  By the $(t,2t)$-property, there is a $2t$-element cocircuit~$C^*$ that contains $\{x_1,\ldots,x_t\}$.

  Suppose that $C^* \neq C$.
  Then there is some $i \in [t]$ such that $y_i\notin C^*$.
  Without loss of generality, say $y_1\notin C^*$.
  Let $I$ be a $(t-1)$-element subset of $[t+1,n]$.
  For any such $I$, the set $S_1 \cup (\bigcup_{i \in I} S_i)$ is a circuit that meets $C^*$.
  By orthogonality, $\bigcup_{i \in I} S_i$ meets $C^*$ for every $(t-1)$-element subset~$I$ of $[t+1,n]$.
  Thus, $C^*$ avoids at most $t-2$ of the $S_i$'s for $i \in [t+1,n]$.
  In fact, as $C^*$ meets each $S_i$ with $i \in [t]$, the cocircuit~$C^*$ avoids at most $t-2$ of the $S_i$'s with $i \in [n]$.
  Thus $|C^*| \ge n-(t-2) \ge (3t-1) -(t-2) =2t+1 > 2t$; a contradiction.
  Therefore, we conclude that $C^*=C$, and the result follows.
\end{proof}

\sloppy
\begin{lemma}
  \label{lem:rep-orthog}
  Let $M$ be a matroid with the $(t,2t)$-property, and let $(S_1,\ldots, S_n)$ be a $t$-echidna of $M$ with $n\geq3t-1$.
  Let $I$ be a $(t-1)$-element subset of $[n]$.
  For $z\in E(M)-\bigcup_{i \in I}S_i$, there is a $2t$-element circuit and a $2t$-element cocircuit each containing $\{z\} \cup (\bigcup_{i \in I}S_i)$.
\end{lemma}
\fussy

\begin{proof}
  For $i \in [n]$, let $S_i=\{x_i,y_i\}$.
  By the $(t,2t)$-property, there is a $2t$-element circuit~$C$ containing $\{z\} \cup \{x_i : i \in I\}$.
  Let $J$ be a $(t-1)$-element subset of $[n]$ such that $C$ and $\bigcup_{j \in J}S_j$ are disjoint (such a set exists since $|C|=2t$ and $n \ge 3t-1$).
  For $i \in I$, let $C_i^*=S_i \cup (\bigcup_{j \in J} S_j)$, and observe that $x_i \in C_i^* \cap C$, and $C_i^* \cap C \subseteq S_i$.
  By \cref{lem:coechidna}, $(S_1,\dotsc,S_n)$ is a $t$-coechidna as well as a $t$-echidna; therefore, $C_i^*$ is a cocircuit.
  Now, for each $i \in I$, orthogonality implies that $|C_i^* \cap C| \ge 2$, and hence $y_i \in C$.
  So $C$ contains $\{z\} \cup (\bigcup_{i \in I}S_i)$, as required.

  By a dual argument, there is also a $2t$-element cocircuit~$C^*$ containing $\{z\} \cup (\bigcup_{i \in I}S_i)$.
\end{proof}

Let $(S_1,\dotsc,S_n)$ be a $t$-echidna of a matroid $M$.
If $(S_1,\dotsc,S_m)$ is a $t$-echidna of $M$, for some $m \ge n$, we say that $(S_1,\dotsc,S_n)$ \emph{extends} to $(S_1,\dotsc,S_m)$.
We say that $\pi=(S_1,\dotsc,S_n)$ is \emph{maximal} if there is no echidna other than $\pi$ to which $\pi$ extends.

\begin{lemma}
  \label{lem:swamping}
  Let $M$ be a matroid with the $(t,2t)$-property, with $t \ge 2$.
  If $M$ has a $t$-echidna $(S_1,\ldots, S_n)$, where $n\geq 4t-3$, then $(S_1,\ldots, S_n)$ extends to a partition of $E(M)$ that is both a $t$-echidna and a $t$-coechidna. 
\end{lemma}

\begin{proof}
  Suppose that $(S_1,\dotsc,S_n)$ extends to $\pi=(S_1, \dotsc, S_m)$, where $\pi$ is maximal.
  Let $X = \bigcup_{i=1}^m S_i$.
  By \cref{lem:coechidna}, $\pi$ is a $t$-coechidna as well as a $t$-echidna.
  The result holds if $X=E(M)$.
  Therefore, towards a contradiction, we suppose that $E(M)-X\neq\emptyset$.
  Let $z\in E(M)-X$.
  By \cref{lem:rep-orthog}, there is a $2t$-element circuit~$C = \{z,z'\} \cup (\bigcup_{i \in [t-1]}S_i)$, for some $z' \in E(M)-\big(\{z\} \cup (\bigcup_{i \in [t-1]}S_i)\big)$.

  We claim that $z' \notin X$.
  Towards a contradiction, suppose that $z'\in S_k$ for some $k\in [t,m]$.
  Let $J$ be a $t$-element subset of $[t,m]$ containing $k$.
  Then, since $(S_1,\dotsc,S_m)$ is a $t$-coechidna, $\bigcup_{j \in J}S_j$ is a cocircuit that contains $z'$.
  Now, by orthogonality, $z \in X$; a contradiction.
  Thus, $z'\notin X$, as claimed.

  We next show that $(\{z,z'\}, S_t, S_{t+1}, \ldots, S_m)$ is a $t$-coechidna.
  It suffices to show that $\{z,z'\} \cup (\bigcup_{i \in I}S_i)$ is a cocircuit for each $(t-1)$-element subset~$I$ of $[t,m]$.
  Let $I$ be such a set.
  \Cref{lem:rep-orthog} implies that there is a $2t$-element cocircuit~$C^*$ of $M$ containing $\{z\} \cup (\bigcup_{i\in I}S_i)$.
  By orthogonality, $|C\cap C^*|>1$. Therefore, $z'\in C^*$. Thus, $(\{z,z'\}, S_t, S_{t+1}, \ldots, S_m)$ is a $t$-coechidna.
  Since this $t$-coechidna has order $1+m-(t-1)\geq3t-1$, the dual of \cref{lem:coechidna} implies that $(\{z,z'\}, S_t, S_{t+1}, \dotsc, S_m)$ is also a $t$-echidna.

  Now, we claim that $(\{z,z'\}, S_1, S_2, \dotsc, S_m)$ is a $t$-coechidna.
  It suffices to show that $\{z,z'\} \cup (\bigcup_{i \in I}S_i)$ is a cocircuit for any $(t-1)$-element subset~$I$ of $[m]$.
  Let $I$ be such a set, and let $J$ be a $(t-1)$-element subset of $[t,m]-I$.
  By \cref{lem:rep-orthog}, there is a $2t$-element cocircuit~$C^*$ containing $\{z\} \cup (\bigcup_{i \in I}S_i)$.
  Moreover, $C=\{z,z'\} \cup (\bigcup_{j \in J}S_j)$ is a circuit since $(\{z,z'\}, S_t, S_{t+1}, \dotsc, S_m)$ is a $t$-echidna.
By orthogonality, $z'\in C^*$. Therefore, $(\{z,z'\}, S_1, S_2, \dotsc, S_m)$ is a $t$-coechidna.
By the dual of \cref{lem:coechidna}, it is also a $t$-echidna, contradicting the maximality of $(S_1,\dotsc,S_m)$. 
\end{proof}

\section{Matroids with the $(t,2t)$-property}
\label{sec:t2t}

In this section, we prove that every sufficiently large matroid with the $(t,2t)$-property is a $t$-spike.
Our primary goal is to show that a sufficiently large matroid with the $(t,2t)$-property has a large $t$-echidna or $t$-coechidna;
it then follows, by \cref{lem:swamping}, that the matroid is a $t$-spike.

For the entirety of the section, we assume that $M$ is a matroid with the $(t,2t)$-property.

\begin{lemma}
  \label{lem:rank-t}
  Let $X\subseteq E(M)$.
  \begin{enumerate}
    \item If $r(X)<t$, then $X$ is independent.\label{rt1}
    \item If $r(X)=t$, then $M|X\cong U_{t,|X|}$.\label{rt2}
  \end{enumerate}
\end{lemma}

\begin{proof}
  Clearly, as $M$ has the $(t,2t)$-property, 
  $M$ has no circuits of size at most $t$.
  Thus, if $r(X)<t$, then $X$ contains no circuits and is therefore independent.
  If $r(X)=t$, then a subset of $X$ is a circuit if and only if it has size $t+1$.
  Therefore, $M|X\cong U_{t,|X|}$.
\end{proof}

\begin{lemma}
  \label{lem:no-Ut3t}
  $M$ has no restriction isomorphic to $U_{t,3t}$.
\end{lemma}

\begin{proof}
  Let $X\subseteq E(M)$, and suppose towards a contradiction that $M|X\cong U_{t,3t}$.
  Let $x\in X$, and let $C^*$ be a cocircuit of $M$ containing $x$.
  Then $E(M)-C^*$ is closed, so $\cl(X-C^*)\subseteq \cl(E(M)-C^*) = E(M)-C^*$.
  Therefore, $r(X-C^*)<r(X)=t$, implying that $|C^*| > 2t$.
  But then every cocircuit containing $x$ has size greater than $2t$, contradicting the $(t,2t)$-property.
\end{proof}

\begin{lemma}
  \label{lemmaA}
  Let $C_1^*,C_2^*,\dotsc,C_{t-1}^*$ be a collection of $t-1$ pairwise disjoint cocircuits of $M$. 
  Let $Y = E(M)-\bigcup_{i \in [t-1]} C_i^*$.
  For all $y \in Y$, there is a $2t$-element circuit~$C_y$ containing $y$ such that either
  \begin{enumerate}
    \item $|C_y \cap C_i^*| = 2$ for all $i \in [t-1]$, or\label{A1}
    \item $|C_y \cap C_j^*| = 3$ for some $j \in [t-1]$, and $|C_y \cap C_i^*| = 2$ for all $i \in [t-1]-\{j\}$.\label{A2}
  \end{enumerate}
  Moreover, if $C_y = S \cup \{y\}$ satisfies \ref{A2}, then there are at most $3t-1$ elements $w \in Y$ such that $S \cup \{w\}$ is a circuit.
\end{lemma}
\begin{proof}
  Choose an element $c_i \in C_i^*$ for each $i \in [t-1]$.  By the $(t,2t)$-property, there is a $2t$-element circuit~$C_y$ containing $\{c_1,c_2,\dotsc,c_{t-1},y\}$, for each $y \in Y$.
  By orthogonality, $C_y$ satisfies \ref{A1} or \ref{A2}.

  Suppose $C_y$ satisfies \ref{A2}, and let $S =C_y-Y= C_y-\{y\}$.
  Let $W = \{w \in Y : S \cup \{w\} \textrm{ is a circuit}\}$.
  It remains to prove that $|W| < 3t$.
  Observe that $W \subseteq \cl(S) \cap Y$, and, since $S$ contains $t-1$ elements in pairwise disjoint cocircuits that avoid $Y$, we have $r(\cl(S) \cup Y) \ge r(Y) + (t-1)$.  Thus,
  \begin{align*}
    r(W) &\le r(\cl(S) \cap Y) \\
    &\le r(\cl(S)) + r(Y) - r(\cl(S) \cup Y) \\
    &\le (2t-1) + r(Y) - (r(Y)+ (t-1)) \\
    &=t,
  \end{align*}
  using submodularity of the rank function at the second line.

  Now, by \cref{lem:rank-t}\ref{rt1}, if $r(W) < t$, then $W$ is independent, so $|W| = r(W) < t$.
  On the other hand, by \cref{lem:rank-t}\ref{rt2}, if $r(W)=t$, then $M|W \cong U_{t,|W|}$.
  Since $M$ has no restriction isomorphic to $U_{t,3t}$, by \cref{lem:no-Ut3t}, we deduce that $|W| < 3t$, as required.
\end{proof}

The next lemma can be viewed as a stronger form of \cref{lem:disjoint2} for a matroid with the $(t,2t)$-property.

\begin{lemma}
  \label{lem:disjoint}
  There exists a function $h: \mathbb{N} \times \mathbb{N} \rightarrow \mathbb{N}$ such that if $M$ has at least $h(\ell,d)$ $\ell$-element circuits, then $M$ has a collection of $d$ pairwise disjoint $2t$-element cocircuits.
\end{lemma}

\begin{proof}
  By \cref{lem:disjoint2}, there is a function $g$ such that if $M$ has at least $g(\ell,d)$ $\ell$-element circuits, then $M$ has a collection of $d$ pairwise disjoint circuits.
  We define $h(\ell,d) = g(\ell,td)$,
  and claim that a matroid with at least $h(\ell,d)$ $\ell$-element circuits has a collection of $d$ pairwise disjoint $2t$-element cocircuits.

  Let $M$ be such a matroid.
  By \cref{lem:disjoint2}, $M$ has a collection of $td$ pairwise disjoint circuits.
  We partition these into $d$ groups of size $t$: 
  call this partition $(\mathcal{C}_1,\dotsc,\mathcal{C}_d)$.
  Since the $t$ circuits in any cell of this partition are pairwise disjoint,
  it now suffices to show that, for each $i \in [d]$, there is a $2t$-element cocircuit contained in the union of the members of $\mathcal{C}_i$.
  Let $\mathcal{C}_i = \{C_1,\dotsc,C_t\}$ for some $i \in [d]$.
  Pick some $c_j \in C_j$ for each $j \in [t]$.
  Then, by the $(t,2t)$-property, $\{c_1,c_2,\dotsc,c_t\}$ is contained in a $2t$-element cocircuit, which, by orthogonality, is contained in $\bigcup_{j \in [t]}C_j$.
\end{proof}

\begin{lemma}
  \label{setup}
  There exists a function $g$ such that if $|E(M)| \ge g(t,q)$, then, for some $M' \in \{M,M^*\}$, the matroid $M'$ has $t-1$ pairwise disjoint cocircuits $C_1^*, C_2^*, \dotsc, C_{t-1}^*$, and
  there is some $Z \subseteq E(M')-\bigcup_{i \in [t-1]}C_i^*$ such that
  \begin{enumerate}
    \item $r_{M'}(Z) \ge q$, and\label{ps1}
    \item for each $z \in Z$, there exists an element $z'\in Z-\{z\}$ such that each $C_i^*$ contains a pair of elements $\{x_i,x_i'\}$ for which $\{z,z'\} \cup (\bigcup_{i\in [t-1]}\{x_i,x_i'\})$ is a circuit of $M'$.
      \label{ps2}
  \end{enumerate}
\end{lemma}

\begin{proof}
  Let $M' \in \{M,M^*\}$.
  By \cref{lem:disjoint}, there is a function $h$ such that if $M'$ has at least $h(\ell,d)$ $\ell$-element circuits, then $M'$ has a collection of $d$ pairwise disjoint $2t$-element cocircuits.

  Suppose $|E(M)|\geq 2t\cdot h(2t,t-1)$.
  Then, by the $(t,2t)$-property, $M'$ has at least $h(2t,t-1)$ distinct $2t$-element circuits.
  Hence, by \cref{lem:disjoint}, $M'$ has a collection of $t-1$ pairwise disjoint $2t$-element cocircuits $C_1^*, C_2^*, \dotsc, C_{t-1}^*$.

  Let $X = \bigcup_{i \in [t-1]}C_i^*$ and $Y=E(M)-X$.
  By \cref{lemmaA}, for each $y \in Y$ there is a $2t$-element circuit~$C_y$ containing $y$ such that $|C_y \cap C_j^*| = 3$ for at most one $j \in [t-1]$ and $|C_y \cap C_i^*| = 2$ otherwise.
  Let $W$ be the set of all $w \in Y$ such that 
  $w$ is in a $2t$-element circuit~$C$ with $|C\cap C_j^*|=3$ for some $j \in [t-1]$, and $|C \cap C_i^*|=2$ for all $i \in [t-1]-\{j\}$.
  Now, letting $Z=Y-W$, we see that \ref{ps2} is satisfied for both $M' = M$ and $M'=M^*$.

  Since the $C_i^*$'s have size $2t$,
  there are $(t-1)\binom{2t}{3}\binom{2t}{2}^{t-2}$ 
  sets $X'\subseteq X$ with $|X' \cap C_j^*|=3$ for some $j \in [t-1]$ and $|X' \cap C_i^*|=2$ for all $i \in [t-1]-\{j\}$.
  It follows, by \cref{lemmaA}, that $|W| \le s(t)$ where $$s(t) = (3t-1)\left[(t-1)\binom{2t}{3}\binom{2t}{2}^{t-2}\right].$$

  We define $$g(t,q) = \max\left\{2t\cdot h(2t,t-1), 2\big(q+s(t)+2t(t-1)\big)\right\}.$$
  Suppose that $|E(M)| \ge g(t,q)$.
  Recall that \ref{ps2} holds for both $M'=M$ and $M'=M^*$.
  Moreover, we can choose $M' \in \{M,M^*\}$ such that $r(M') \ge q+s(t)+2t(t-1)$. Then,
  \begin{align*}
    r_{M'}(Z) &\ge r_{M'}(Y) - |W| \\
    &\ge \big(r(M')-2t(t-1)\big) - s(t) \\
    &\ge q,
  \end{align*}
  so \ref{ps1} holds as well, as required.
\end{proof}

\sloppy
\begin{lemma}
  \label{payoff}
  Suppose $M$ has $t-1$ pairwise disjoint cocircuits $C_1^*, C_2^*, \dotsc, C_{t-1}^*$,
  and, for some positive integer~$p$, there is some $Z \subseteq E(M)-\bigcup_{i \in [t-1]}C_i^*$ such that
  \begin{enumerate}[label=\rm(\alph*)]
    \item $r(Z) \ge \binom{2t}{2}^{t-1}(p + 2(t-1))$, and
%
    \item for each $z \in Z$, there exists an element $z'\in Z-\{z\}$ such that each $C_i^*$ contains a pair of elements $\{x_i,x_i'\}$ for which $\{z,z'\} \cup (\bigcup_{i\in [t-1]}\{x_i,x_i'\})$ is a circuit of $M$.
  \end{enumerate}
  Then there exists a subset $Z' \subseteq Z$ and a partition $\mathcal{Z}' = ( Z_1', \dotsc, Z_p' )$ of $Z'$ into pairs such that
  \begin{enumerate}
    \item each circuit of $M|Z'$ is a union of pairs in $\mathcal{Z}'$, and
    \item the union of any $t$ pairs of $\mathcal{Z}'$ contains a circuit.
  \end{enumerate}
\end{lemma}
\fussy

\begin{proof}
  We first prove the following:

  \begin{sublemma}
    \label{prepayoff}
    There exists a set $X=\bigcup_{i \in [t-1]}\{x_i,x_i'\}$, with $\{x_i,x_i'\} \subseteq C_i^*$, and a collection $\mathcal{Z}' = \{Z_1', \dotsc, Z_p'\}$ of $p$ disjoint pairs of elements of $Z$ such that
    \begin{enumerate}[label=\rm(\Roman*)]
      \item $X \cup Z_i'$ is a circuit, for each $i \in [p]$, and\label{ppo1}
      \item $\mathcal{Z}'$ partitions the ground set of $(M/X)|Z'$ into parallel classes, and $r_{M/X}\big(\bigcup_{i \in [p]}Z_i'\big)=p$.\label{ppo2}
    \end{enumerate}
  \end{sublemma}
  \begin{subproof}
    For each $z \in Z$, there exists an element $z'\in Z-\{z\}$ such that $\{z,z'\} \cup (\bigcup_{i\in [t-1]}\{x_i,x_i'\})$ is a circuit of $M$, where $\{x_i,x_i'\} \subseteq C_i^*$.
    For any selection of pairs $\{x_i,x_i'\} \subseteq C_i^*$ for each $i \in [t-1]$, there is a set $X = \bigcup_{i \in [t-1]}\{x_i,x_i'\}$ and some $Z'\subseteq Z$ such that for each $z \in Z'$, there is an element $z'\in Z'$ such that $X \cup \{z,z'\}$ is a circuit.
%
    There are $\binom{2t}{2}^{t-1}$ choices for $X$, each with a corresponding set $Z'$. 
    Let $m = \binom{2t}{2}^{t-1}$.
    Observe that the union of the $m$ possibilities for $Z'$ is $Z$.
    For $m$ sets $Z_1,\dotsc,Z_{m}$ whose union is $Z$, we have that $r(Z_1) + \dotsb + r(Z_m) \ge r(Z)$.
    Thus, by the pigeonhole principle, there exists some $X$ and $Z'$ with $$r(Z') \ge \frac{r(Z)}{\binom{2t}{2}^{t-1}} \ge p+2(t-1).$$
    
    Now, observe that $X \cup \{z,z'\}$ is a circuit, for some pair $\{z,z'\} \subseteq Z'$, if and only if $\{z,z'\}$ is a parallel pair in $M/X$.
    So the ground set of $(M/X)|Z'$ has a partition into parallel classes, where each parallel class has size at least two.
    Let $\mathcal{Z}' = \{\{z_1,z_1'\}, \dotsc,\{z_n,z_n'\}\}$ be a collection of pairs from each parallel class such that $\{z_1,z_2,\dotsc,z_n\}$ is independent in $(M/X)|Z'$.
    Since $r_{M/X}(Z') = r(Z' \cup X) -r(X) \ge r(Z') - 2(t-1) \ge p$, there exists such a collection $\mathcal{Z}'$ of size $p$, and this collection satisfies \ref{prepayoff}.
  \end{subproof}

  Let $X$ and $\mathcal{Z}'=\{Z_1',\dotsc,Z_p'\}$ be as described in \cref{prepayoff}, let $Z' = \bigcup_{i \in [p]} Z_i'$ and let $\mathcal{X} = \{X_1,\dotsc,X_{t-1}\}$, where $X_i = \{x_i,x_i'\} = X \cap C_i^*$.

  \begin{sublemma}
    \label{metamatroid}
    Each circuit of $M|(X \cup Z')$ is a union of pairs in $\mathcal{X} \cup \mathcal{Z}'$.
  \end{sublemma}
  \begin{subproof}
    Let $C$ be a circuit of $M|(X \cup Z')$.
    If $x_i \in C$, for some $\{x_i,x_i'\} \in \mathcal{X}$, then, by orthogonality with $C_i^*$, we have $x_i' \in C$.
    Towards a contradiction, say $\{z,z'\} \in \mathcal{Z}'$ and $C \cap \{z,z'\} = \{z\}$.
    Choose $W$ to be the union of the pairs of $\mathcal{Z}'$ that contain elements of $(C-\{z\}) \cap Z'$.  Then $z \in \cl(X \cup W)$.  Hence $z \in \cl_{M/X}(W)$, contradicting \cref{prepayoff}\ref{ppo2}.
  \end{subproof}

  \begin{sublemma}
    \label{induct}
    The union of any $t$ pairs of $\mathcal{X} \cup \mathcal{Z}'$ contains a circuit.
  \end{sublemma}
  \begin{subproof}
    Let $\mathcal{W}$ be a subcollection of $\mathcal{X} \cup \mathcal{Z}'$ of size $t$.
    We proceed by induction on the number of pairs in $\mathcal{W} \cap \mathcal{Z}'$.
    If there is only one pair in $\mathcal{W} \cap \mathcal{Z}'$, then the union of the pairs in $\mathcal{W}$ contains a circuit (indeed, is a circuit) by \cref{prepayoff}\ref{ppo1}.
    Suppose the result holds for any subcollection containing $k$ pairs in $\mathcal{Z}'$, and let $\mathcal{W}$ be a subcollection containing $k+1$ pairs in $\mathcal{Z}'$.
    Let $\{x,x'\}$ be a pair in $\mathcal{X}-\mathcal{W}$,
    and let $W = \bigcup_{W' \in \mathcal{W}}W'$.
    By the induction hypothesis, $W \cup \{x,x'\}$ contains a circuit~$C_1$.
    If $\{x,x'\} \subseteq E(M)-C_1$, then $C_1 \subseteq W$, in which case the union of the pairs in $\mathcal{W}$ contains a circuit, as desired.
    Therefore, we may assume, by \cref{metamatroid}, that $\{x,x'\} \subseteq C_1$.
    Since $X$ is independent, there is a pair $\{z,z'\} \subseteq Z' \cap C_1$.
    By the induction hypothesis, there is a circuit~$C_2$ contained in $(W-\{z,z'\}) \cup \{x,x'\}$.
    Observe that $C_1$ and $C_2$ are distinct, and $\{x,x'\} \subseteq C_1 \cap C_2$.
    By circuit elimination on $C_1$ and $C_2$, and \cref{metamatroid}, there is a circuit $C_3 \subseteq (C_1 \cup C_2) - \{x,x'\} \subseteq W$, as desired.  The result now follows by induction.
  \end{subproof}

  Now, \cref{induct} implies that the union of any $t$ pairs of $\mathcal{Z}'$ contains a circuit, and the result follows.
\end{proof}

In order to prove \cref{mainthm}, we use some hypergraph Ramsey Theory~\cite{ramsey1930}.

\begin{theorem}[Ramsey's Theorem for $k$-uniform hypergraphs]
  \label{hyperramsey}
  For positive integers $k$ and $n$, there exists an integer $r_k(n)$ such that if $H$ is a $k$-uniform hypergraph on $r_k(n)$ vertices, then $H$ has either a clique on $n$ vertices, or a stable set on $n$ vertices.
\end{theorem}

We now prove \cref{mainthm}, restated below as \cref{mainthmtake2}.

\begin{theorem}
  \label{mainthmtake2}
  There exists a function $f : \mathbb{N} \rightarrow \mathbb{N}$ such that if $M$ is a matroid with the $(t,2t)$-property, and $|E(M)| \ge f(t)$, then $M$ is a $t$-spike.
\end{theorem}
\begin{proof}
  We first consider the case where $t=1$.  Let $M$ be a non-empty matroid with the $(1,2)$-property.  Then, for every $e \in E(M)$, the element~$e$ is in a parallel pair $P$ and a series pair $S$.  By orthogonality, $P=S$, and $P$ is a connected component of $M$. 
  Then $M \cong U_{1,2} \oplus M\ba P$, and the result easily follows.

  We may now assume that $t \ge 2$.
  We define the function $h_k : \mathbb{N} \rightarrow \mathbb{N}$, for each $k \in [t]$, as follows:
  $$h_k(t) = \begin{cases}
      4t-3 & \textrm{if $k= t$,}\\
      r_k(h_{k+1}(t)) & \textrm{if $k \in [t-1]$,}
    \end{cases}$$
    where $r_k(n)$ is the Ramsey number described in \cref{hyperramsey}.
  Note that $h_{k}(t) \ge h_{k+1}(t) \ge 4t-3$, for each $k \in [t-1]$.
  Let $p(t) = h_1(t)$ and let $q(t) = \binom{2t}{2}^{t-1}(p(t) + 2(t-1))$.

  By \cref{setup}, there exists a function $g$ such that if $|E(M)| \ge g(t,q(t))$, then,
  for some $M' \in \{M,M^*\}$, the matroid $M'$ has $t-1$ pairwise disjoint cocircuits $C_1^*, C_2^*, \dotsc, C_{t-1}^*$, and there is some $Z' \subseteq E(M')-\bigcup_{i \in [t-1]}C_i^*$ such that
      $r_{M'}(Z') \ge q(t)$, and,
      for each $z \in Z'$, there exists an element $z'\in Z'-\{z\}$ such that $\{z,z'\} \cup (\bigcup_{i\in [t-1]}\{x_i,x_i'\})$ is a circuit of $M'$, where $\{x_i,x_i'\} \subseteq C_i^*$.

  Let $f(t) = g(t,q(t))$, and suppose that $|E(M)| \ge f(t)$.
  For ease of notation, we assume that $M' = M$.
  Then, by \cref{payoff},
  there exists a subset $Z \subseteq Z'$ and a partition $\mathcal{Z} = ( Z_1, \dotsc, Z_{p(t)} )$ of $Z$ into $p(t)$ pairs such that
  \begin{enumerate}[label=\rm(\Roman*)]
    \item each circuit of $M|Z$ is a union of pairs in $\mathcal{Z}$, and
    \item the union of any $t$ pairs of $\mathcal{Z}$ contains a circuit.\label{rc2}
  \end{enumerate}

  By \cref{lem:swamping}, and since $t \ge 2$, it suffices to show that $M$ has a $t$-echidna or a $t$-coechidna of order $4t-3$.
  If the smallest circuit in $M|Z$ has size $2t$, then, by \ref{rc2}, $\mathcal{Z}$ is a $t$-echidna of order $p(t) \ge 4t-3$.  So we may assume that the smallest circuit in $M|Z$ has size $2j$ for some $j \in [t-1]$.
  \begin{sublemma}
    \label{iterramsey}
    If the smallest circuit in $M|Z$ has size $2j$, for $j \in [t-1]$, and $|\mathcal{Z}| \ge h_j(t)$, then either
    \begin{enumerate}
      \item $M$ has a $t$-coechidna of order $4t-3$, or\label{ir1}
      \item there exists some $Z' \subseteq Z$ that is the union of $h_{j+1}(t)$ pairs of $\mathcal{Z}$ for which the smallest circuit in $M|Z'$ has size at least $2(j+1)$.\label{ir2}
    \end{enumerate}
  \end{sublemma}
  \begin{subproof}
    Let $2j$ be the size of the smallest circuit in $M|Z$.
    We define $H$ to be the $j$-uniform hypergraph with vertex set $\mathcal{Z}$ whose hyperedges are the $j$-subsets of $\mathcal{Z}$ that are partitions of circuits in $M|Z$.
    By \cref{hyperramsey}, and the definition of $h_k$, as $H$ has at least $h_j(t)$ vertices, it has either a clique or a stable set, on $h_{j+1}(t)$ vertices.
    If $H$ has a stable set~$\mathcal{Z}'$ on $h_{j+1}(t)$ vertices, then clearly \ref{ir2} holds, with $Z' = \bigcup_{P \in \mathcal{Z}'} P$.

    So we may assume that there are $h_{j+1}(t)$ pairs in $\mathcal{Z}$ such that the union of any $j$ of these pairs is a circuit.
    Let $Z''$ be the union of these $h_{j+1}(t)$ pairs.
    We claim that the union of any set of $t$ pairs contained in $Z''$ is a cocircuit.
    Let $T$ be a transversal of $t$ pairs of $\mathcal{Z}$ contained in $Z''$, and let $C^*$ be the $2t$-element cocircuit containing $T$.
    Towards a contradiction, suppose that there exists some pair $P \in \mathcal{Z}$ with $P \subseteq Z''$ such that $|C^* \cap P| = 1$.
    Select $j-1$ pairs $Z_1'',\dotsc,Z_{j-1}''$ of $\mathcal{Z}$ that are each contained in $Z''-C^*$ (these exist since $h_{j+1}(t) \ge 3t-1 \ge 2t + j - 1$).
    Then $P \cup (\bigcup_{i \in [j-1]}Z_i'')$ is a circuit that intersects the cocircuit~$C^*$ in a single element, contradicting orthogonality.
    We deduce that the union of any $t$ pairs of $\mathcal{Z}$ that are contained in $Z''$ is a cocircuit.
    So $M$ has a $t$-coechidna of order $h_{j+1}(t) \ge 4t-3$, satisfying \ref{ir1}.
  \end{subproof}

  We now apply \cref{iterramsey} iteratively, for a maximum of $t-j$ iterations.
  If \ref{ir1} holds, at any iteration, then $M$ has a $t$-coechidna of order $4t-3$, as required.
  Otherwise, we let $\mathcal{Z}'$ be the partition of $Z'$ induced by $\mathcal{Z}$; then, at the next iteration, we relabel $Z=Z'$ and $\mathcal{Z}=\mathcal{Z}'$. 
  If \ref{ir2} holds for each of $t-j$ iterations, then we obtain a subset $Z'$ of $Z$ such that the smallest circuit in $M|Z'$ has size $2t$.
  Then, by \ref{rc2}, $M$ has a $t$-echidna of order $h_{t}(t) = 4t-3$.
  This completes the proof.
\end{proof}

\section{Properties of $t$-spikes}
\label{sec:tspikeprops}

In this section, we prove some properties of $t$-spikes, which demonstrate that $t$-spikes form a class of highly structured matroids.
In particular,
we show that a $t$-spike has order at least $2t-1$; a $t$-spike of order~$r$ has $2r$ elements and rank~$r$; the circuits of a $t$-spike that are not a union of $t$ arms meet all but at most $t-2$ of the arms; and a $t$-spike of order at least $4t-4$ is $(2t-1)$-connected.
We also show that an appropriate concatenation of the associated partition of a $t$-spike is a $(2t-1)$-anemone, 
following the terminology of~\cite{ao2008}.

It is straightforward to see that the family of $1$-spikes consists of matroids obtained by taking direct sums of copies of $U_{1,2}$.
We 
also
describe a construction that can be used to obtain a $(t+1)$-spike 
from a $t$-spike, 
and show that every $(t+1)$-spike can be constructed from some $t$-spike in this way.


\subsection*{Basic properties}

\begin{lemma}
  \label{tspikeorder}
  Let $M$ be a $t$-spike of order~$r$.  Then $r \ge 2t-1$.
\end{lemma}
\begin{proof}
  Let $(A_1,\ldots,A_r)$ be the associated partition of $M$.
  By definition, $r \ge t$.
  Let $J$ be a $t$-element subset of $[r]$, and let $Y = \bigcup_{j \in J}A_j$.
  Pick some $y \in Y$.
  Since $Y$ is a cocircuit and a circuit, $Z=(E(M)-Y) \cup \{y\}$ spans and cospans $M$.  
  Since $|Z| = 2(r-t)+1$,
  $$2r = |E(M)| = r(M) + r^*(M) \le \big(2(r-t)+1\big) + \big(2(r-t)+1\big).$$
  It follows that $r \ge 2t-1$.
\end{proof}

\begin{lemma}
  Let $M$ be a $t$-spike of order~$r$. 
  Then $r(M) = r^*(M) = r$.
\end{lemma}
\begin{proof}
  Let $(A_1,\ldots,A_r)$ be the associated partition of $M$, and label $A_i = \{x_i,y_i\}$ for each $i \in [r]$.
  Pick $I \subseteq J \subseteq [r]$ such that $|I|=t-1$ and $|J| = r-t$.
  Let $X = \left(\bigcup_{i \in I}A_i\right) \cup \{x_j : j \in J\}$, and observe that $|X| = |I| + |J| = r-1$.
  Now, since $(A_1,\ldots,A_r)$ is a $t$-echidna, $\bigcup_{j \in J}A_j \subseteq \cl(X)$.
  As $E(M)-\bigcup_{j \in J}A_j$ is a cocircuit, we deduce that $r(M)-1 \le r(X) \le |X| = r-1$, so $r(M) \le r$.
  Similarly, as $(A_1,\ldots,A_r)$ is a $t$-coechidna, we deduce that $r^*(M) \le r$.
  Since $r(M) + r^*(M) = |E(M)| = 2r$, the lemma follows. 
\end{proof}

The next lemma shows that a circuit~$C$ of a $t$-spike is either a union of $t$ arms, or else $C$ meets all but at most $t-2$ of the arms.
\sloppy
\begin{lemma}
  \label{l:circuits}
  Let $M$ be a $t$-spike of order~$r$ with associated partition $(A_1,\ldots,A_r)$, and
  let $C$ be a circuit of $M$. 
  Then either
  \begin{enumerate}
    \item 
      $C = \bigcup_{j \in J}A_j$ for some $t$-element set $J \subseteq [r]$, or\label{c1}
    \item 
      $\left|\{i \in [r] : A_i \cap C \neq \emptyset\}\right| \ge r-(t-2)$ and
      $\left|\{i \in [r] : A_i \subseteq C\}\right| < t$.\label{c2}
  \end{enumerate}
\end{lemma}
\fussy
\begin{proof}
  Let $S = \{i \in [r] : A_i \cap C \neq \emptyset\}$, so $S$ is the minimal subset of $[r]$ such that $C \subseteq \bigcup_{i \in S}A_i$.
  If $C$ is properly contained in $\bigcup_{j \in J}A_j$ for some $t$-element set $J \subseteq [r]$, then $C$ is independent; a contradiction.
  So $|S| \ge t$.
  If $|S|=t$, then $C = \bigcup_{i \in S}A_i$, implying $C$ is a circuit, which satisfies \ref{c1}.
  So we may assume that $|S| > t$.
  Now $\left|\{i \in [r] : A_i \subseteq C\}\right| < t$,
  otherwise $C$ properly contains a circuit.
  Thus, there exists some $j \in S$ such that $A_j - C \neq \emptyset$.
  If $|S| \ge r-(t-2)$, then \ref{c2} holds;
  so we assume that $|S| \le r-(t-1)$.
  Let $T = ([r]-S) \cup \{j\}$.  Then $|T|\ge t$, so $\bigcup_{i \in T}A_i$ contains a cocircuit that intersects $C$ in one element, contradicting orthogonality.
\end{proof}

\subsection*{Connectivity}

Let $M$ be a matroid with ground set $E$.
Recall that
the \emph{connectivity function} of $M$, denoted by $\lambda$, is defined as 
\begin{align*}
  \lambda(X) = r(X) + r(E - X) - r(M),
\end{align*}
for all subsets $X$ of $E$.
%
%
It is easily verified that
\begin{align}
  \label{eq:conn2}
  \lambda(X) = r(X) + r^*(X) - |X|.
\end{align}
%

A subset $X$ or a partition $(X, E-X)$ of $E$ is \emph{$k$-separating} if $\lambda(X) < k$.
A $k$-separating partition $(X,E-X)$ is a \emph{$k$-separation} if $|X| \ge k$ and $|E-X|\ge k$.
The matroid~$M$ is \emph{$n$-connected} if, for all $k < n$, it has no $k$-separations.

\sloppy
\begin{lemma}
  \label{l:conn}
  Suppose $M$ is a $t$-spike with associated partition $(A_1,\ldots,A_r)$.
  Then, for all partitions $(J,K)$ of $[r]$ with $|J| \le |K|$,
    $$\lambda\left(\bigcup_{j \in J} A_j\right) = 
    \begin{cases}
      2|J| & \textrm{if $|J| < t$,}\\
      2t-2 & \textrm{if $|J| \ge t$.}
    \end{cases}$$
\end{lemma}
\fussy
\begin{proof}
  Let $(J,K)$ be a partition of $[r]$ with $|J|\le |K|$.
  \begin{sublemma}
    \label{c:1}
    The lemma holds when $|J| \le t$.
  \end{sublemma}
  \begin{subproof}
    Suppose $|J| < t$.
    Since $(A_1,\ldots,A_r)$ is a $t$-echidna (respectively, $t$-coechidna), 
    $\bigcup_{j \in J} A_j$ is independent (respectively, coindependent).
    So, by \eqref{eq:conn2}, $\lambda\big(\bigcup_{j \in J} A_j\big) = 2|J| + 2|J| - 2|J| = 2|J|$.

    Now suppose $|J|=t$.  Then, by definition, $\bigcup_{j \in J} A_j$ is a circuit and a cocircuit.
    So $\lambda\big(\bigcup_{j \in J} A_j\big) = (2t-1) + (2t-1) - 2t = 2t-2$, by \eqref{eq:conn2}.
  \end{subproof}


  \begin{sublemma}
    \label{c:2}
    Let $X \subseteq Y \subseteq [r]$ such that $|X| \ge t-1$. 
    Then $$\lambda\left(\bigcup_{x \in X} A_x\right) \ge \lambda\left(\bigcup_{y \in Y} A_y\right).$$
  \end{sublemma}
  \begin{subproof}
    Let $X'$ be a $(t-1)$-element subset of $X$, and let $y \in Y-X$.
    Then $\lambda\big(\bigcup_{x \in X'} A_x\big) = 2(t-1)$, and
    $\lambda\big(A_y \cup (\bigcup_{x \in X'} A_x)\big) = 2t-2$, by \ref{c:1}.
    By submodularity of the connectivity function,
    \begin{align*}
      \lambda\left(A_y \cup \bigcup_{x \in X}A_x\right) &\le \lambda\left(A_y \cup \bigcup_{x \in X'}A_x\right) + \lambda\left(\bigcup_{x \in X}A_x\right) - \lambda\left(\bigcup_{x \in X'}A_x\right) \\
      &=(2t-2) + \lambda\left(\bigcup_{x \in X}A_x\right) -(2t-2) \\
      &=\lambda\left(\bigcup_{x \in X}A_x\right).
    \end{align*}
  \ref{c:2} now follows by induction.
  \end{subproof}
  
  Now suppose $|J| > t$.
  By \ref{c:1} and \ref{c:2}, $\lambda\big(\bigcup_{j \in J}A_j\big) \le 2t-2$. 
  Recall that $|K| \ge |J| > t$.
  Let $K'$ be a $t$-element subset of $K$.
  Let $J' = [r]-K'$, and note that $J \subseteq J'$.
  So, by \cref{c:2}, $$\lambda\left(\bigcup_{j \in J}A_j\right) \ge
  \lambda\left(\bigcup_{j \in J'}A_j\right) = \lambda\left(\bigcup_{k \in K'}A_k\right) = 2t-2.$$
  We deduce that $\lambda\big(\bigcup_{j \in J}A_j\big)=2t-2$, as required.
\end{proof}

Given a $t$-spike~$M$ with associated partition $(A_1,\dotsc,A_r)$, suppose that $(P_1,\dotsc,P_m)$ is a partition of $E(M)$ such that, for each $i \in [m]$, $P_i = \bigcup_{i \in I}A_i$ for some subset $I$ of $[r]$, with $|P_i| \ge 2t-2$.
Using the terminology of~\cite{ao2008}, it follows immediately from \cref{l:conn} that $(P_1,\dotsc,P_m)$ is a $(2t-1)$-anemone. 
(Note that a partition whose concatenations gives rise to a flower in this way have previously appeared in the literature~\cite{gw2013} under the name of ``quasi-flowers''.)

\begin{lemma}
  Let $M$ be a $t$-spike of order at least $4t-4$, for $t \ge 2$.
  Then $M$ is $(2t-1)$-connected.
\end{lemma}
\begin{proof}
  Let $r$ be the order of the $t$-spike~$M$, 
  and let $(A_1,\ldots,A_r)$ be the associated partition of $M$.
  Towards a contradiction, suppose $M$ is not $(2t-1)$-connected, and let $(P,Q)$ be a $k$-separation for some $k < 2t-1$.
  Without loss of generality, we may assume that $|P| \ge |Q|$.
  Note, in particular, that $\lambda(P) < k \le |Q|$ and $\lambda(P) < 2t-2$.

  Suppose $|P \cap A_j| \neq 1$ for all $j \in [r]$.
  Then, by \cref{l:conn}, $\lambda(P) = |Q|$ if $|Q| < 2t$, otherwise $\lambda(P) = 2t-2$;
  either case is contradictory.
  So
  $|P \cap A_j| = 1$ for some $j \in [r]$.

  Suppose $|Q| \le 2t-2$.  Then, by \cref{l:circuits} and its dual, $Q$ is independent and coindependent, so $\lambda(P) = |Q|$ by \eqref{eq:conn2}; 
  a contradiction.

  Now we may assume that $|Q| > 2t-2$.
  Suppose $\bigcup_{i \in I}A_i \subseteq P$, for some $(t-1)$-element set $I \subseteq [r]$.
  Then $A_j \subseteq \cl(P)$ for each $j \in [r]$ such that $|P \cap A_j| = 1$.
  For such a $j$, it follows, by the definition of $\lambda$, that $\lambda(P \cup A_j) \le \lambda(P)$; we use this repeatedly in what follows.
  Let $U = \{u \in [r] : |P \cap A_u|= 1\}$.
  For any subset $U' \subseteq U$, we have $\lambda\left(P \cup (\bigcup_{u \in U'}A_u)\right) \le \lambda(P) < 2t-2$.
  Let $P' = P \cup (\bigcup_{u \in U}A_u)$, and let $Q' = E(M)-P'$.
  If $|Q'| > 2t-2$, then $\lambda(P') = 2t-2$ by \cref{l:conn}, contradicting that $\lambda(P') \le \lambda(P) < 2t-2$.
  So $|Q'| \le 2t-2$.
  Now, let $d = |Q|-(2t-2)$, and let $U'$ be a $d$-element subset of $U$.
  Then $\lambda(P) \ge \lambda\left(P \cup (\bigcup_{u \in U'}A_u)\right) = \lambda\left(Q- \bigcup_{u \in U'}A_u\right)$.  Since $\left|Q- \bigcup_{u \in U'}A_u\right| = 2t-2$, we have that $\lambda\left(Q- \bigcup_{u \in U'}A_u\right)=2t-2$, so $\lambda(P) \ge 2t-2$; a contradiction.
  We deduce that $|\{i \in [r] : A_i \subseteq P\}| < t-1$.
  Since $|Q| \le |P|$, it follows that $|\{i \in [r] : A_i \subseteq Q\}| \le |\{i \in [r] : A_i \subseteq P\}| < t-1$.

  Now $|\{i \in [r] : A_i \cap Q \neq \emptyset\}| \ge r-(t-2)$, so $r(Q) \ge r-(t-1)$ by \cref{l:circuits}.
  Similarly, $r(P) \ge r-(t-1)$.
  So 
  \begin{align*}
    \lambda(P) &= r(P) + r(Q) - r(M) \\
    &\ge (r-(t-1)) + (r-(t-1)) - r \\
    &\ge (4t-4) -2(t-1) = 2t-2;
  \end{align*}
  a contradiction.
  This completes the proof.
\end{proof}

\subsection*{Constructions}

We first describe a construction that can be used to obtain a $(t+1)$-spike of order~$r$ from a $t$-spike of order~$r$, when $r \ge 2t+1$.
We then show that every $(t+1)$-spike can be constructed from some $t$-spike in this way.

Recall that $M_1$ is an \emph{elementary quotient} of $M_0$ if there is a single-element extension $M^+_0$ of $M_0$ by an element~$e$ such that $M_1 = M^+_0 / e$.
A matroid $M_1$ is an \emph{elementary lift} of $M_0$ if $M_1^*$ is an elementary quotient of $M_0^*$.
Note also that if $M_1$ is an elementary quotient of $M_0$, then $M_0$ is an elementary lift of $M_1$. 

Let $M_0$ be a $t$-spike of order~$r \ge 2t+1$ with associated partition $\pi$.
Let $M_0'$ be an elementary quotient of $M_0$ such that none of the $2t$-element cocircuits are preserved (that is, extend $M_0$ by an element~$e$ that blocks all of the $2t$-element cocircuits, and then contract $e$).
Now, in $M_0'$, the union of any $t$ cells of $\pi$ is still a $2t$-element circuit, but, as $r(M_0') = r(M_0) - 1$, the union of any $t+1$ cells of $\pi$ is a $2(t+1)$-element cocircuit.
We then repeat this in the dual;
that is, let $M_1$ be an elementary lift of $M_0'$ such that none of the $2t$-element circuits are preserved. 
Then $M_1$ is a $(t+1)$-spike.
Note that $M_1$ is not unique; more than one $(t+1)$-spike can be constructed from a given $t$-spike $M_0$ in this way.


Given a $(t+1)$-spike~$M_1$, for some positive integer~$t$, we now describe how to obtain a $t$-spike~$M_0$ from $M_1$ by a specific elementary quotient, followed by a specific elementary lift. This process reverses the construction from the previous paragraph.
The next lemma describes the single-element extension (or coextension, in the dual) that gives rise to the elementary quotient (or lift) we desire.
Intuitively, the extension adds a ``tip'' to a $t$-echidna.
In the proof of this lemma, we assume knowledge of the theory of modular cuts (see \cite[Section~7.2]{oxbook}).

\begin{lemma}
  \label{modcut}
  Let $M$ be a matroid with a $t$-echidna $\pi=(S_1,\dotsc,S_n)$. 
  Then there is a single-element extension~$M^+$ of $M$ by an element~$e$ such that $e \in \cl_{M^+}(X)$ if and only if $X$ contains at least $t-1$ spines of $\pi$, for all $X \subseteq E(M)$.
\end{lemma}
\begin{proof}
  Let $$\mathcal{F} = \left\{\bigcup_{i\in I}S_i : I \subseteq [n] \textrm{ and } |I|=t-1\right\}.$$ 
  By the definition of a $t$-echidna, $\mathcal{F}$ is a collection of flats of $M$.
  Let $\mathcal{M}$ be the set of all flats of $M$ containing some flat $F \in \mathcal{F}$.
  We claim that $\mathcal{M}$ is a modular cut.
  Recall that, for distinct $F_1,F_2 \in \mathcal{M}$, the pair $(F_1,F_2)$ is \emph{modular} if $r(F_1) + r(F_2) = r(F_1 \cup F_2) + r(F_1 \cap F_2)$.
  It suffices to prove that for any $F_1,F_2 \in \mathcal{M}$ such that $(F_1,F_2)$ is a modular pair, $F_1 \cap F_2 \in \mathcal{M}$.

  For any $F \in \mathcal{M}$, since $F$ contains at least $t-1$ spines of $\pi$, and the union of any $t$ spines is a circuit (by the definition of a $t$-echidna), it follows that $F$ is a union of spines of $\pi$.
  So let $F_1,F_2 \in \mathcal{M}$ such that $F_1=\bigcup_{i\in I_1}S_i$ and $F_2=\bigcup_{i\in I_2}S_i$, where $I_1$ and $I_2$ are distinct subsets of $[n]$ with $u_1=|I_1| \ge t-1$ and $u_2=|I_2|\ge t-1$.
  Then
  \begin{align*}
    r(F_1) + r(F_2) &= (t-1 + u_1) + (t-1 + u_2) \\
    &=2(t-1) + u_1 + u_2.
  \end{align*}

  Suppose that $|I_1 \cap I_2| < t-1$. 
  Let $s=|I_1 \cap I_2|$.
  Then $F_1 \cup F_2$ is the union of $u_1 + u_2 - s \ge t-1$ spines of $\pi$.
  So \begin{align*}
    r(F_1 \cup F_2) + r(F_1 \cap F_2) &= \big(t-1 + (u_1 + u_2 - s)\big) +2s \\
    &= (t-1)+s+u_1+u_2.
  \end{align*}
  Since $s < t-1$, it follows that $r(F_1 \cup F_2) + r(F_1 \cap F_2) < r(F_1) + r(F_2)$. 
  So, for every modular pair $(F_1,F_2)$ with $F_1,F_2 \in \mathcal{M}$, we have $|I_1 \cap I_2| \ge t-1$, in which case $F_1 \cap F_2$ is a flat containing the union of $t-1$ spines of $\pi$, and hence $F_1 \cap F_2 \in \mathcal{M}$ as required.

  Now, there is a single-element extension corresponding to the modular cut~$\mathcal{M}$, and this extension satisfies the requirements of the lemma (see, for example, \cite[Theorem~7.2.3]{oxbook}).
%
%
\end{proof}

Let $M$ be a $t$-spike with associated partition $\pi=(A_1,\dotsc,A_r)$, for some integer $t \ge 2$, where $r \ge 2t-1$ by \cref{tspikeorder}.
Let $M^+$ be the single-element extension of $M$ by an element~$e$ described in \cref{modcut}.
%

Consider $M^+/e$.
We claim that $\pi$ is a $(t-1)$-echidna and a $t$-coechidna of $M^+/e$.
Let $X$ be the union of any $t-1$ spines of $\pi$.  Then $X$ is independent in $M$, and 
$X \cup \{e\}$ is a circuit in $M^+$, so $X$ is a circuit in $M^+/e$.
So $\pi$ is a $(t-1)$-echidna of $M^+/e$.
Now let $C^*$ be the union of any $t$ spines of $\pi$, and let $H=E(M)-C^*$.  Then $H$ is the union of at least $t-1$ spines, so $e \in \cl_{M^+}(H)$.  Now $H \cup \{e\}$ is a hyperplane in $M^+$, so $C^*$ is a cocircuit in $M^+$.
Hence $\pi$ is a $t$-coechidna of $M^+/e$.

We now repeat this process on $N=(M^+/e)^*$.
In $N$, the partition $\pi$ is a $t$-echidna and $(t-1)$-coechidna.
By \cref{modcut}, there is a single-element extension $N^+$ of $N$ (a single-element coextension of $M^+/e$) by an element $e'$. 
By the same argument as in the previous paragraph,
%
%
$\pi$ is a $(t-1)$-echidna and $(t-1)$-coechidna of $N^+/e$, so $N^+/e$ is a $(t-1)$-spike.
Let $M' = (N^+/e)^*$.

Note that $M^+/e$ is an elementary quotient of $M$, so $M$ is an elementary lift of $M^+/e$ where none of the $2(t-1)$-element circuits of $M^+/e$ are preserved in $M$.
Similarly, $M^+/e$ is an elementary quotient of $M'$ where none of the $2(t-1)$-element cocircuits are preserved.
So the $t$-spike $M$ can be obtained from the $(t-1)$-spike $M'$ using the earlier construction.



\sloppy
\bibliographystyle{abbrv}
\bibliography{lib}

\end{document}